\documentclass[12pt]{elsarticle}
\usepackage{amsmath,amssymb,amsfonts,amscd,amsthm}
\usepackage[unicode]{hyperref}
\newtheorem{thm}{Theorem}[section]
\newtheorem{cor}[thm]{Corollary}
\newtheorem{lem}[thm]{Lemma}
\newtheorem{keylem}[thm]{Key Lemma}
\newtheorem{prop}[thm]{Proposition}
                   \theoremstyle{definition}
\newtheorem{property}[thm]{Property}
\newtheorem{defn}[thm]{Definition}

\newcommand\blank{\mathord{\hbox to 1.5ex{\hrulefill}}\,}
\DeclareMathOperator\Ker{Ker}
\DeclareMathOperator\Max{Max}
\DeclareMathOperator\GL{GL}
\DeclareMathOperator\Hom{Hom}
\DeclareMathOperator\GG{G}

\DeclareMathOperator\E{E}
\DeclareMathOperator\SL{SL}
\DeclareMathOperator\Um{Um}

\DeclareMathOperator\K{K}

\DeclareMathOperator\rank{rank}

\DeclareMathOperator\EE{EE}
\newcommand\q{\mathfrak q}

\newcommand\Z{\mathbb Z}
\newcommand\N{\mathbb N}

\let\ph\varphi
\let\mf\mathfrak
\let\wt\widetilde

\newcommand\hlrarrow{\lhook \mkern-8mu\hbox to 28pt{\rightarrowfill}}
\newcommand\tensor{\mathbin{\overline\otimes}}

\journal{Journal of Algebra}

\begin{document}

\begin{frontmatter}
\title
{Structure of Chevalley groups over rings via universal localization}

\author[spbgu,leti]{Alexei Stepanov}
\ead{stepanov239@gmail.com}
\address[spbgu]{Dept. of Mathematics and Mechanics,
St.Petersburg State University,\\
Universitetsky 28, Peterhof,
198504, St. Petersburg, Russia}
\address[leti]
{St.Petersburg State Electrotechnical University,\\
Prof. Popova 5, 197376, St. Petersburg, Russia}

\begin{abstract}
In the current article we study the structure of a Chevalley group $G(R)$ over a
commutative ring $R$. We generalize and improve the following results on:
\begin{itemize}
\item
the standard, relative, and multi-relative commutator formulas;
\item
the nilpotent structure of [relative] $\K_1$;
\item
the bounded word length of commutators.
\end{itemize}
To this end we enlarge the elementary group, construct a generic element for the
extended elementary group, and use localization in the universal ring. The key
step is a construction of a generic element for the principal congruence
subgroup, corresponding to a principal ideal.
\end{abstract}


\end{frontmatter}

\begin{center}
\itshape
Dedicated to my teacher and good friend, Nikolai Vavilov,\\ on his sixtieth birthday
\end{center}

\section*{Introduction}

In the current article we study the structure of a Chevalley group $G(R)$ over a commutative
ring $R$ with a reduced irreducible root system $\Phi$. We generalize and improve the following results:
\begin{itemize}
\item
the standard, relative, and multi-relative commutator formulas;
\item
the nilpotent structure of a [relative] $\K_1$;
\item
the bounded word length of commutators.
\end{itemize}
The results of the article were announced in~\cite{StepAnnounce}.

The standard commutator formulas for $G=\GL_n$ were proved by L.\,Vaser\-stein in~\cite{VasersteinGLn}
and independently by Z.\,Borewich and N.\,Vavilov in~\cite{BV84}.
In 1987 G.\,Taddei proved the normality of the elementary subgroup in a Chevalley
group. It was already known at that time how to deduce the standard commutator formulas from this
result using the double of a ring along an ideal and the splitting principle.
The relative commutator formula was proved in different settings and by different methods
in~\cite{VavStepCommForm0,VavStepCommForm,HazZhangCommForm,HVZrelachev}.
In~\cite{VavStepCommForm} it was shown that in most cases it follows from the
standard commutator formulas by pure group theoretical arguments.
However, the proof of multiple relative commutator formula for $G=\GL_n$
from~\cite{HazZhangMulti} used all the power of Bak's localization procedure.

In the current article for all simply connected Chevalley groups we prove that
\begin{align*}
&[\wt E(R,\mf a),G(R,\mf b)]\le\EE(R,\mf a,\mf b),\text{ where}\\
&\EE(R,\mf a,\mf b)=E(R,\mf a\mf b)[E(R,\mf a),E(R,\mf b)],
\end{align*}
$\mf a$ and $\mf b$ are ideals of $R$, and $\wt E$ denotes the extended elementary
subgroup defined in~\ref{defEtilde}.
The formula is shown to be stronger than all the commutator formulas mentioned above.
Note that if $\Phi\ne C_l,G_2$ or $2$ is invertible in $R$, then
$\EE(R,\mf a,\mf b)=[E(R,\mf a),E(R,\mf b)]$, and if $\mf a+\mf b=R$, then
$\EE(R,\mf a,\mf b)=E(R,\mf a\mf b)$.

\smallskip
Let $Y$ be a subset of an abstract group $H$ and let $\Sigma$ be a generating set of~$H$.
The \emph{width} (or word length) of $Y$ with respect to $\Sigma$ is the smallest integer
$L$ such that each element of $Y$ can be written as a product of at most $L$ generators
and their inverses. The width of an element $a$ is the width of one-point set $\{a\}$, the width
of the identity element equals $0$.

The width of the linear elementary group $\E_n(R)$ with respect to the set of
elementary transvections or to the set of all commutators was studied by D.\,Carter,
G.\,Keller, K.\,Dennis, L.\,Vaserstein,
W.~van der Kallen, O.\,Tavgen and others. This is related to computing the Kazhdan
constatant~\cite{Shalom}.
The width of $E(R)$ is known to be finite if $R$ is semilocal (by Gauss decomposition)
or $R=\mathbb O_S$, where $O_S$ is the ring of $S$-integers of
a number field and $\Phi\ne A_1$
(by D.\,Carter and G.\,Keller~\cite{CarterKellerZ} and O.\,Tavgen~\cite{TavgenChevalley}).
But it is infinite for $\E_n(\mathbb C[x])$ (W.~van der Kallen~\cite{vdkSL3}).
It is amazing that the answer is unknown
already for $\E_n(F[x])$, where $F$ is a finite field or a number field.

It turns out that finiteness of the width of the elementary subgroup $E(R)$ of a
Chevalley group with respect to the elementary root unipotents is equivalent to
finiteness of the width of $E(R)$ with respect to commutators
(provided that $\Phi\ne C_2,G_2$ or $R$ has no residue fields of $2$ elements, otherwise elementary group
is not necessarily perfect). This follows from joint results of the author with
A.\,Sivatski~\cite{SivStep} (for $G=\SL_n$) and with
N.\,Vavilov~\cite{StepVavLength} (for all simply connected Chevalley groups of rank $\ge 2$).
It is shown in these papers that the width of the set of all commutators is
finite with respect to the elementary root unipotents provided that
the maximal spectrum of the ground  ring $R$ has finite combinatorial
dimension. The width bound obtained in these articles depended on the dimension.

In the current article we prove that the width of a commutator $[a,b]$,
$a\in G(R,\mf a)$, $b\in\wt E(R,\mf b)$ with respect to an arbitrary functorial generating
set of $\EE(R,\mf a,\mf b)$ is bounded by a constant, depending only on the root system of $G$.

One of the motivations to study width of commutators is the following observations by
W.~van der Kallen: the group $\E_n(R)^\infty/E_n(R^\infty)$ is an obstruction
for the finiteness of the width of $\E_n(R)$, where infinite power means the direct
product of countably many copies of a ring or a group.
The bounded width of commutators implies that this group is central in
$K_1(R^\infty)$.

\smallskip
Recall that $K_1$-functor of a Chevalley group $G$ is defined by $K_1^G=G(R)/E(R)$.
If $G=\GL_n$ and we pass to the limit, we get the definition
$K_1(R)=\lim\limits_{n\to\infty}K_1^{\GL_n}(R)$  of usual stable $K_1$.
The group $K_1(R)$ is known to be abelian for an arbitrary (even noncommutative)
ring $R$, see~\cite{BassBook}. Nonstable group $K_1^G(R)$ can be nonabelian;
examples can be found in~\cite{vdkUniRows83,BakNonabelian}. But for a finite
dimensional Noetherian ring $R$ the group $K_1^G(R)$ is
nilpotent provided that $G$ is simply connected.
A nilpotent filtration was obtained by A.\,Bak in~\cite{BakNonabelian} for $G=\SL_n$ and
by R.\,Hazrat and N.\,Vavilov in~\cite{HazVav} for all Chevalley groups.
A relative version appeared in~\cite{BHVstrike} by the three authors.

We generalize all results on the nilpotent structure of $K_1^G$ mentioned above
by proving the following commutator formula:
$$
[G(R,\mf a_0),G(R,\mf a_1),\dots,G(R,\mf a_m)]\le\EE(R,\mf a_0\dots\mf a_{m-1},\mf a_m),
$$
where the Bass--Serre dimension of $R$ is not too big with respect to $m$.

\smallskip
All the results of the current article are proved uniformly, without
computations with individual elements of Chevalley groups.
The proof is based on a new version of localization method called universal
localization. The idea is fairly simple:
\begin{itemize}
\item
define an extended elementary subgroup;
\item
construct a generic element for this subgroup;
\item
use localization procedure in a versal ring (unlike ``universal'', the term ``versal'' does not require uniqueness);
\item
project the result to an arbitrary ring.
\end{itemize}
Since localization procedure is used only in a versal ring, we call the method
``universal localization''.

Our proofs are much easier than the previous ones.
One of the main concerns of the article is to show that one can avoid computations with
two denominators called by N.\,Vavilov the ``yoga of commutators''.
Of course, to get real bounds for the width of commutators
one needs some computations, but it seems that bounds obtained
via ``yoga of commutators'' developed in~\cite{StepVavLength,HSVZyoga,HVZrelachev,HSVZwidth}
are far from beeing optimal anyway.
From computational point of view the article is a direct continuation of~\cite{StepComput};
we prove all results (except Gauss decomposition) which are not proved
in~\cite{StepComput} beginning with the standard commutator formulas.
Thus, this branch of the theory of Chevalley groups over rings
(Suslin's local-global principle, the normality of elementary subgroup, the nilpotent structure
of $K_1$, bounded width of commutators, and the like)
requires only computation of the span
$\langle U_P(\mf a),U^-_P(\mf b)\rangle$ of the unipotent
radicals of two opposite parabolic subgroups over ideals $\mf a,\mf b$.
Note that all these results are independent of the root system $\Phi\ne A_1$ and invertibility
of structure constants.

Another branch is the normal structure, which requires elementary subgroup to be perfect
as well as computation of levels. Both results depend on the root system and invertibility of
structure constants. A key computation for this branch was made by M.\,Stein~\cite{SteinChevalley}
in~1971. In the next article we plan to prove the normal structure of Chevalley
groups using only Stein's computation and the standard commutator formulas.

\section{Notation}

\subsection*{Groups}
The identity element of a group will be denoted by $1$. However, if $G$ is a group scheme
and $R$ is a ring, then we denote by $e_R$ the identity element of the group $G(R)$.
Let $a,b,c$ be elements of a group $H$. By $a^b=b^{-1}ab$ we denote the element conjugate to
$a$ by $b$. The commutator $aba^{-1}b^{-1}$ is denoted by $[a,b]$.
Multi-commutators are left normed, i.\,e.
$\bigl[a_1,\dots,a_m\bigr]=\bigl[[a_1,\dots,a_{m-1}],a_m\bigr]$, where $a_1,\dots,a_m\in H$.
The following commutator identity will be used in the sequel.

\begin{lem}\label{CommId}
$[a,\prod\limits_{k=1}^l b_k]=\prod\limits_{k=1}^l[a,b_k]^{c_k}$, where
$c_k=\left(\prod\limits_{i=1}^{k-1}b_i\right)^{-1}$.
\end{lem}

Let $A$ and $B$ be subsets of $H$. By $A^B$ we denote the normal closure of $A$ by $B$, i.\,e
the smallest subgroup of $H$, containing $A$ and normalized by $B$.
The mixed commutator subgroup $[A,B]$ is a subgroup of $H$, generated by all commutators
$[a,b]$, $a\in A$ and $b\in B$. If $A=\{a\}$ is a singleton, then we drop braces and
write $[a,B]$ instead of $[\{a\},B]$. Mixed multi-commutator subgroups are left-normed.

\subsection*{Rings}
Throughout the article the term ``ring'' stands for a commutative ring with a unit
and all ring homomorphisms preserve the units.

Let $S$ be a multiplicative (i.\,e. multiplicatively closed) subset of a ring $R$.
By $S^{-1}R$ we denote the localization of $R$ with respect to $S$. The localization
homomorphism $R\to S^{-1}R$ is denoted by $\lambda_S$. If $S=\{s^k\mid k\in\N\}$,
then we write $\lambda_s:R\to R_s$ for the localization homomorphism.

A sequence $(s_1,\dots,s_m)$ of elements of $R$ is called unimodular
if it generates the unit ideal. The set of all unimodular sequences of length $m$ over $R$
is denoted by $\Um_m(R)$.

An ideal $\mf a$ of a ring $C$ is called a splitting ideal if $C=R\oplus\mf a$ as additive groups,
where $R$ is a subring of $C$. Of course, in this case $R\cong C/\mf a$. Equivalently,
$\mf a$ is a splitting ideal iff
it is a kernel of a retraction $C\to R\subseteq C$. For instance, if $C=R[t]$ is a polynomial ring,
then $tC$ is a splitting ideal. Another example of a splitting ideal is the fundamental
ideal $I$ of the affine algebra $A$ of a group scheme $G$.
Note that if $\mf a$ is a splitting ideal of a ring $C=R\oplus\mf a$,
then $\mf a\otimes_R C'$ is a splitting ideal of $C\otimes_R C'$ for any $R$-algebra $C'$.

Let $K$ be a ring and let $\ph:R\to R''$ and $\ph':R'\to R''$ be $K$-algebra homomorphisms.
Denote by $\iota:R\to R\otimes_KR'$ and $\iota':R'\to R\otimes_KR'$ the canonical homomorphisms.
By the universal property of tensor product there exists a unique $K$-algebra homomorphism
$\ph'':R\otimes_KR'\to R''$ such that $\ph=\ph''\circ\iota$ and $\ph'=\ph''\circ\iota'$.
This homomorphism will be denoted by $\ph\tensor_K\ph'$
to distinguish it from the natural homomorphism
$\ph\otimes_K\ph':R\otimes_KR'\to R''\otimes_KR''$.
Clearly, $\ph\tensor_K\ph'=\operatorname{mult}\circ(\ph\otimes_K\ph')$, where ``$\operatorname{mult}$''
is the multiplication homomorphism. In the case $K=\Z$ we write $\otimes$ and $\tensor$ instead of
$\otimes_K$ and $\tensor_K$, respectively.

\section{Group schemes}\label{schemes}

\textit{Throughout the article the expression ``group scheme'' means ``flat affine group scheme of finite type''}.
Let $K$ be a ring and $G$ an affine group scheme over $K$.
Denote by $A=K[G]$ the affine algebra of~$G$.
By definition of an affine scheme, an element $h\in G(R)$ can be identified with a ring
homomorphism $h:A\to R$. We always perform this identification.
Denote by $g\in G(A)$ the generic element of $G$, i.\,e. the identity map $A\to A$.
An element $h\in G(R)$ induces a group homomorphism
$G(h):G(A)\to G(R)$ by the rule $G(h)(a)=h\circ a$. Thus, the image of $g$
under $G(h)$ equals~$h$. In the sequel, for a ring homomorphism $\ph:R\to R'$ we denote by
the same symbol $\ph$ the induced homomorphism $G(\ph):G(R)\to G(R')$.
This cannot lead to a confusion as we always can distinguish between two different meanings
of $\ph$ by the type of its argument. With this convention we have
$h(g)=h\circ\operatorname{id}_A=h$.

It is easy to see that the identity element $e_K\in G(K)$ coincides with the counit map $A\to K$.
Let $I$ denote the fundamental ideal, i.\,e. the kernel of $e_K$. Then, $g\in G(A,I)=\Ker G(e_K)$.
Note that $e_K$ is a retraction of the structure map $K\to A$, hence $A=K\oplus I$ as
additive groups.
Let $\q$ be an ideal of a ring $R$ and $\rho_\q:R\to R/\q$ the reduction homomorphism.
The kernel of the induced map $\rho_\q:G(R)\to G(R/\q)$ is called the principal
congruence subgroup and is denoted by $G(R,\q)$. It is easy to see that $h\in G(R,\q)$
iff the following diagram commutes.
$$
\begin{CD}
A  @>h>>   R   \\
@V{e_K}VV  @VV{\rho_\q}V\\
K  @>>>    R/\q\\
\end{CD}
$$
And the latter is obviously equivalent to saying that $h(I)\subseteq\q$. In particular,
$\q=h(I)$ is the smallest ideal such that $h\in G(R,\q)$.

\begin{lem}\label{GG}
Let $\mf a$ and $\mf b$ be ideals of a ring $R$, Then
$[G(R,\mf a),G(R,\mf b)]\le G(R,\mf a\mf b)$.
\end{lem}

\begin{proof}
A proof for $G=\GL_n$ can be found in~\cite{VavStepCommForm0} and the
general case can be easily deduced from this. However, we prefer to give a direct proof
using the formalism of the current section.

Note that
$A\otimes_KA=(K\oplus I)\otimes_K(K\oplus I)=K\oplus I\otimes_KK\oplus K\otimes_KI\oplus I\otimes_KI$.
Therefore, $(A\otimes_KI)\cap (I\otimes_KA)=I\otimes_KI$.
There are two natural inclusions of the ring $A$ into $A\otimes_KA$. Denote by
$g_1$ and $g_2$ the images of the generic element $g\in G(A)$ in $G(A\otimes_KA)$
under these inclusions. Since, $g_1\in G(A\otimes_KA,I\otimes_KA)$ and
$g_2\in G(A\otimes_KA,A\otimes_KI)$, we have
\begin{multline*}
[g_1,g_2]\in G\bigl(A\otimes_KA,I\otimes_KA\bigr)\cap G\bigl(A\otimes_KA,A\otimes_KI\bigr)=\\
G\bigl(A\otimes_KA,(I\otimes_KA)\cap (A\otimes_KI)\bigr)=G\bigl(A\otimes_KA,I\otimes_KI\bigr).
\end{multline*}

Now, let $a\in G(R,\mf a)$ and $b\in G(R,\mf b)$. Then
the map $a\otimes b:A\otimes_KA\to R$ takes $I\otimes_KA$ into $\mf a$ and $A\otimes_KI$ into $\mf b$.
The induced group homomorphism $a\otimes b:G(A\otimes_KA)\to G(R)$ maps
$g_1$ to $a$ and $g_2$ to $b$. Therefore,
$$
[a,b]=(a\otimes b)\bigl([g_1,g_2]\bigr)\in (a\otimes b)\bigl(G(A\otimes_KA,I\otimes_KI)\bigr)\le
G(R,\mf a\mf b).
$$
\end{proof}

\section{Chevalley groups}
From now on $G$ denotes a Chevalley--Demazure group scheme over $\Z$
with a reduced irreducible root system $\Phi\ne A_1$.
Throughout the article we keep the notation of the previous section:
$A=\Z[G]$, $g\in G(A)$ is the generic element, and $I$ is the fundamental ideal of $A$.
Denote by $E$ the elementary subgroup (subfunctor) of $G$.
If $G$ is simply connected, then it follows from Gauss decomposition that
over a semilocal ring $R$ the group $E(R)$ coincides with $G(R)$, see~\cite[Corollary~2.4]{AbeSuzuki}
or~\cite{SmolSuryVav,SmolSuryVavBis} for a more general result.
From the very beginning we choose a split maximal torus $T$; all root subgroups $X_\alpha$ are
assumed to be $T$-invariant.

Let $\q$ be an ideal of a ring $R$.
By $E(\q)$ we denote the subgroup of $E(R)$, generated by $x_\alpha(r)$ for all
$\alpha\in\Phi$ and $r\in\q$. The relative elementary subgroup $E(R,\q)$ is the
normal closure of $E(\q)$ in $E(R)$.

The following result relates two groups defined above. It was first stated (without a proof)
by J.\,Tits in~\cite{TitsCongruence}; a proof appeared in the paper~\cite{VasersteinChevalley}
by L.\,Vaserstein. A more detailed proof using the same idea can be found in~\cite{StepComput}.
All computations with individual elements of a Chevalley group are hidden in this proof.
Dilation principle and normality of elementary subgroup follows easily from this result.
We state only a particular case $\q=tR$ which will be used in the sequel.

\begin{lem}\label{EE}
Let $t\in R$. Then $E(R,t^3R)\le E(tR)$.
\end{lem}

Another important prerequisite is the splitting principle. It seems that idea to use splitting belongs
to A.\,Suslin. However, in~\cite{SuslinSerreConj} he used another consequence of splitting than
what we use in the current article. The splitting principle stated below was published
in~\cite{AbeLGP} by E.\,Abe, in~\cite{PetrovStavrovaIsotropic} by
V.\,Petrov and A.\,Stavrova (for non-split groups), and
in~\cite{ApteChattRao} by H.\,Apte, P.\,Chattopadhyay, and R.\,Rao,
The following statement is the relative version of the splitting principle from~\cite{ApteStep}.

\begin{lem}\label{splitting}
Let $R$ be a subring of $C$ and $\mf a$ an ideal of $C$ such that $C=R\oplus\mf a$.
Let $\mf b'$ be an ideal of $R$ and $\mf b=\mf b'C$.
Then
$$
\E(\Phi,C,\mf b)\cap\GG(\Phi,C,\mf a)=E(C,\mf a\mf b)[E(C,\mf a),E(C,\mf b)].
$$

In particular, If $\mf b=C$, then $\E(\Phi,C)\cap\GG(\Phi,C,\mf a)=E(C,\mf a)$.
\end{lem}

For the group in the right hand side of the displayed formula above we introduce special notation.
Let $\mf a$ and $\mf b$ be arbitrary ideals of a ring $C$. Then
$$
\EE(C,\mf a,\mf b)=E(C,\mf a\mf b)[E(C,\mf a),E(C,\mf b)].
$$
If $\Phi\ne C_l,G_2$ or $2$ is invertible in $C$, then $E(C,\mf a\mf b)\le[E(C,\mf a),E(C,\mf b)]$
(see~\cite[subsection~1.2]{StepComput}),
hence $\EE(C,\mf a,\mf b)=[E(C,\mf a),E(C,\mf b)]$.

\section{Gauss decomposition}

Fix a split maximal torus $T$ in $G$ and a Borel subgroup, containing this torus.
Let $U$ and $U^-$ denote the unipotent radicals of this Borel subgroup and its opposite.
Denote by $W$ the Weyl group of the root system $\Phi$.
Let $N=N_G(T)$ be the scheme theoretic normalizer of $T$, i.\,e. the
largest subscheme of $G$ such that $N(R)$ normalizes $T(R)$ for any ring $R$
(it exists by~\cite[Exp.~11, Corollaire~5.3~bis]{SGA}).

Details for the following issue can be found in~\cite[II, 1.4 and I, 8.1-8.5]{Jantzen}.
The quotient $N/T$ is isomorphic to the finite group scheme associated to $W$, i.\,e.
$N/T(R)$ can be identified with a certain subset of the group algebra $RW$, containing $W$
(if there are no nontrivial idempotents in $R$, then $N/T(R)=W$).
An element $\dot w\in N(R)$ is called a representative of an element $w\in W$
if the canonical image of $\dot w$ in $N/T(R)$ equals $w$.
Clearly, any two representatives of $w\in W$ in $N(R)$ differ by an element from $T(R)$.
Therefore, the Gauss cell $\mf G_w=wTUU^-$ does not depend on a representative of $w$, and a given
element $a\in\mf G_w$ can be written as $a=\dot wuv$ for some
$\dot w\in N(R)$, $u\in U(R)$, and $v\in U^-(R)$.

\begin{lem}\label{Gauss}
There exist a unimodular set of elements $\{s_w\in A\mid w\in W\}$
such that $\lambda_{s_w}(g)\in w T(A_{s_w})\,U(A_{s_w})\,U^-(A_{s_w})$
for all $w\in W$.
\end{lem}

\begin{proof}
By~\cite[II, 1.9]{Jantzen}
the main cell $TUU^-$ is a principal open subscheme and the Gauss cells
$\mf G_w=w TUU^-$ form an open cover of $G$ as $w$ ranges over $W$.
Since a Gauss cell is a shift of the main cell, all of them are
principal open subschemes, i.\,e. $\mf G_w$ is an affine scheme with
affine algebra $\Z[\mf G_w]=A_{s_w}$ for some $s_w\in A$.
By the definition of an open cover~\cite[I, 1.7(5)]{Jantzen} the set
$\{s_w\mid w\in W\}$ is unimodular.

The image $\lambda_{s_w}(g)$ of $g$ in $G(A_{s_w})$ is the localization homomorphism
$\lambda_{s_w}$ itself. This is a tautology that it factors through $\lambda_{s_w}$ and
this means that $\lambda_{s_w}(g)\in\mf G_w(A_{s_w})$, i.\,e. we have
the required factorization for $\lambda_{s_w}(g)$.
\end{proof}

Of course, being an open cover does not mean that the union of the sets of points of all
Gauss cells coincides with the group of points of the scheme $G$.
If $R$ is a ring, then an element
$a\in G(R)$ lies in $\mf G_w(R)$ iff $a(s_w)$ is invertible in $R$.
Thus, we have neither existence nor uniqueness of a Gauss cell, containing a given element.

\begin{lem}\label{e}
Suppose that $R$ is a nontrivial ring.
The identity element $e_R$ of the group $G(R)$ belongs to exactly one Gauss cell
$\mf G_1$ (here 1 is the identity element of $W$).
\end{lem}

\begin{proof}
Clearly $e_R\in\mf G_1$. Suppose that $e_R$ belongs to another Gauss cell $\mathfrak G_w$.
Choose a residue field $F$ of $R$. Then $e_F$
lies in $\mf G_w(F)=\dot w \mf G_1(F)$, so ${\dot w}^{-1}$ lies in $\mathfrak G_1(F)$ and $w=1$
by uniqueness of Bruhat decomposition in $G(F)$.
\end{proof}

Decomposition of $a\in\mf G_w(R)$ as a product $a=\dot w uv$, where
$\dot w$ is an appropriate preimage of $w$ in $N(R)$, $u\in U(R)$, and $v\in U^-(R)$, is unique.
This implies the following useful property of the main cell $\mf G_1$.

\begin{lem}\label{BigCell}
Let $\q$ be an ideal of a ring $R$. Then $\mf G_1(R)\cap G(R,\q)=T(R,\q)U(R,\q)U^-(R,\q)$.
\end{lem}

\begin{proof}
Let $a=huv\in G(R,\q)$, where $h\in T(R)$, $u\in U(R)$, and $v\in U^-(R)$.
Then $\rho_\q(huv)=e$, hence $\rho_\q(hu)=\rho_\q(v)^{-1}$. Since
$TU\cap U^-$ and $T\cap U$ both are trivial, we have $\rho_\q(h)=\rho_\q(u)=\rho_\q(v)^{-1}=e$,
i.\,e. $h\in T(R,\q)$, $u\in U(R,\q)$, and $v\in U^-(R,\q)$ as required.
\end{proof}

If $G$ is simply connected, then its torus is generated by the images of the coroots,
see~\cite[II, 1.6(4) and 1.3]{Jantzen}.
Since $T_{\SL_2}(R,\q)\le\E_2(R,\q)$ (this is the simplest case of the Whitehead lemma),
we have $T(R,\q)\le E(R,\q)$ for an arbitrary simply connected Chevalley--Demazure group
scheme~$G$.

\begin{cor}\label{radical}
Suppose that an ideal $\q$ of a ring $R$ is contained in the Jacobson radical of $R$.
If $G$ is simply connected, then $G(R,\q)\le E(R,\q)$.
\end{cor}

\begin{proof}
Denote by $J$ the Jacobson radical of $R$.
If $a\in G(R,J)$, then $\rho_J(a)=e\in\mf G_1(R/J)$. It follows that $\rho_J(a)(s_1)=\rho_J\bigl(a(s_1)\bigr)$
is invertible in $R/J$, hence $a(s_1)$ is invertible in $R$, which implies that
$a\in\mf G_1(R)$. By the previous lemma $a\in T(R,\q)U(R,\q)U^-(R,\q)$. Since $G$ is simply connected,
we have $a\in E(R,\q)$.
\end{proof}

\begin{cor}\label{sc-loc}
If $G$ is simply connected, then there exists a unimodular set of elements $\{s_1,\dots,s_l\}$
such that $\lambda_{s_k}(g)\in E(A_{s_k},I_{s_k})$ for all $k=1,\dots,l$.
\end{cor}

\begin{proof}
Take the unimodular set from Lemma~\ref{Gauss} and let $l=\#W$.
There exists a representative of $w$ in $E(A_{s_w})$.
Since $G$ is simply connected, $T\le E$. Therefore all Gauss cells lie inside
the elementary group.

Fix $w\in W$ and let $r=s_w$. Then
$\rho_{I_r}\bigl(\lambda_r(g)\bigr)=\lambda_r\bigl(\rho_I(g)\bigr)=e\in\mf  G_w(A_r/I_r)$.
By Lemma~\ref{e} $e_R$ lies in a Gauss cell $\mf  G_w(R)$ iff $w=1$ or $R$ is the trivial
ring. It follows that $I_r=A_r$ iff $w\ne 1$, i.\,e. $s_w\in I$ for all $w\ne1$.
In particular, $\lambda_{s_w}(g)\in E(A_{s_w},I_{s_w})$ for all $w\ne1$.
For $w=1$ the result follows from Lemma~\ref{BigCell}.
\end{proof}

Normality of the elementary subgroup follows easily from Corollary~\ref{sc-loc}
and Lemma~\ref{EE}, see~\cite{StepComput}.

\section{Standard commutator formulas}

For the results proved in the current article we do not need the elementary subgroup $E(R)$ to be perfect,
therefore we state the standard commutator formulas only as inclusions.
These inclusions turn into equalities if $\Phi\ne C_2,G_2$ or $R$ has no residue fields of two
elements (see~\cite[Corollary 4.4]{SteinChevalley}.
Simultaneously with the commutator formulas we prove that a certain set of commutators
has finite width with respect to an arbitrary functorial generating set, although this result will be a
special case of Theorem~\ref{CommForm}. This is done to illustrate advantages of working with
``generic example'' at the very beginning.

In this section $G$ is a Chevalley--Demazure group scheme, not necessarily simply connected.

\begin{prop}\label{stcommform}
Let $\q$ be an ideal of a ring $R$. Then
$$
[E(R,\q),G(R)]\le E(R,\q)\qquad\text{and}\qquad [E(R),G(R,\q)]\le E(R,\q).
$$
\end{prop}

\begin{proof}
Let  $t$ be an independent variable. For $\alpha\in\Phi$ consider
the element $h=[x_\alpha(t),g]\in G(A[t])$. Since the elementary subgroup is normal, we have
$h\in E(A[t])$. On the other hand, $h$ vanishes modulo $t$ and modulo $I[t]$.
Since both $tA[t]$ and $I[t]$ are splitting ideals, it follows from splitting principle~\ref{splitting}
that $h\in E(A[t])\cap G(A[t],tA[t])=E(A[t],tA[t])$ and $h\in E(A[t])\cap G(A[t],I[t])=E(A[t],I[t])$.

Now, let $a=[x_\alpha(r),b]$ for some $r\in R$ and $b\in G(R)$.
Let $\ph:A[t]\to R$ be the (unique) extension of $b:A\to R$, sending $t$ to $r$.
We have $\ph(g)=b(g)=b$, hence $\ph(h)=a$. If $r\in\q$, then
$a=\ph(h)\in\ph\bigl(E(A[t],tA[t])\bigr)=E(R,\q)$. Similarly, if $b\in G(R,\q)$,
then $b(I)\subseteq\q$ and $a=\ph(h)\in\ph\bigl(E(A[t],I[t])\bigr)\le E(R,\q)$.
\end{proof}

To define the notion of functorial generating set of $E(R,\q)$ we define the category of ideals.
Objects of this category are pairs $(R,\q)$, where $\q$ is an ideal of a ring $R$.
A morphism $\theta:(R,\q)\to(R',\q')$ is a ring homomorphism $R\to R'$ that maps $\q$ into $\q'$.
Let $\Sigma$ be a functor from the category of ideals to the category of sets.
We call it a \emph{functorial generating set} for the relative elementary subgroup if
there is a natural inclusion of $\Sigma(R,\q)$ to the underlying set of $E(R,\q)$
and $\Sigma(R,\q)$ spans $E(R,\q)$.

An example of a functorial generating set of the relative elementary subgroup is well known
(see~\cite{VasersteinChevalley}). Namely,
$$
\Sigma(R,\q)=\{x_\alpha(p)^{x_{-\alpha}(r)}\mid \alpha\in\Phi,\ p\in\q,\ r\in R\}.
$$

\begin{cor}\label{length0}
Let $\Sigma$ be a functorial generating set for the relative elementary subgroup of $G$.
Then there exists $L\in\N$ such that for any ring $R$ and any ideal $\q$ of $R$
the width of the set
$$
\{[x_\alpha(r),b]\mid r\in\q,\, b\in G(R)\text{ or }r\in R,\,b\in G(R,\q)\}
$$
with respect to $\Sigma(R,\q)$ is not greater than $L$.
\end{cor}

\begin{proof}
Clearly, if $\ph:(R,\q)\to(R',\q')$ takes $a\in E(R,\q)$ to $b\in E(R',\q')$,
then the width of $b$ with respect to $\Sigma(R',\q')$ is not greater than the width of
$a$ with respect to $\Sigma(R,\q)$. Now the result follows from the previous proof as
any element of the set under consideration is an image of the element
$h\in E(A[t],I[t])\cap E(A[t],tA[t])$, constructed there.
\end{proof}

\section{Key construction}\label{KeyConstruction}

In this section we construct a ring $D=D_G$ together with elements $t\in D$ and $f\in G(D,tD)$
satisfying the following property.

\begin{property}\label{univ}
Given a ring $R$ and elements $r\in R$ and $h\in G(R,rR)$ there exists a
homomorphism $\theta:D\to R$ such that $\theta(t)=r$ and $\theta(f)=h$.
\end{property}

In case of $G=\SL_n$ the construction is very simple. Namely,
$$
D_{\SL_n}=\Z[t,y_{ij}\mid 1\le i,j\le n]/\bigl(\det(e+ty)-1\bigr),
$$
where $e$ denotes the identity matrix and $y$ the matrix with $y_{ij}$ in
position $(i,j)$.
The general construction is a kind of a cone over $G$ at the identity element
but instead of an embedding to a projective space we use a faithful representation.

Let $K$ be a ring.
A (rational, finite dimensional) representation of an algebraic group $G$ over $K$ is a morphism
$\pi:G\to\GL_n$ of group schemes (a natural transformation of functors
$\pi:\Hom(K[G],\blank)\to\Hom(K[\GL_n],\blank)$ that respects the group structure).
Note that this definition is more restrictive than one in~\cite[I, 2.7]{Jantzen}.
A representation is called faithful if it is a monomorphism
of schemes. It is easy to see that this is equivalent to saying that
all maps $\pi_R:G(R)\to\GL_n(R)$ are injective.
The following fact is well known for specialists.

\begin{lem}\label{faithful}
For any affine group scheme over $\Z$ there exists a faithful representation.
\end{lem}

\begin{proof}
A $\Z$-group scheme $G$ is flat iff $\Z[G]$ is a flat $\Z$-module, which means
that $\Z[G]$ is torsion free. The action of $G(R)$ on itself by (right) multiplications
induces a structure of $G$-module on the affine algebra $\Z[G]$, see~\cite[I, 2.7]{Jantzen}.
Recall that all affine schemes in the article are assumed to be of finite type.
Therefore, $\Z[G]$ is a finitely generated
$\Z$-algebra. Choose a generating set and denote by $M$ the smallest $G$-submodule of $\Z[G]$,
containing all the generators. By~\cite[I, 2.13(3)]{Jantzen} $M$ is finitely
generated over $\Z$. Being a finitely generated torsion free $\Z$-module $M$ is free.
A choice of a basis of $M$ defines a representation $\pi:G\to\GL_n$, where $n$ is the rank
of $M$.

It is straightforward to check that for any ring $R$ two different elements $a,b\in G(R)$
act differently on $R[G]=\Z[G]\otimes R$. Therefore, they act differently on the generators
of $\Z[G]$. It follows that $\pi_R$ is injective for any ring $R$, hence $\pi$
is faithful.
\end{proof}

Denote by $M_n(R)$ the full matrix ring over $R$, so that $M_n$ defines an
affine scheme over $\Z$ isomorphic to the $n^2$-dimensional affine space.
The affine algebra $\Z[M_n]$ will be identified with the polynomial ring
$\Z[z_{ij}\mid 1\le i,j\le n]$. Usually we compose a representation $\pi:G\to\GL_n$
with the natural embedding $\GL_n\to M_n$. By abuse of notation the morphism
of affine schemes $G\to M_n$ will be called a representation and denoted by the
same symbol $\pi$. Clearly, the map $\pi:G\to M_n$ is a monomorphism iff
$G\to\GL_n$ is so.

Let $\mf R$ denote the category of commutative rings with units.
A representation $\pi:G\to M_n$ defines a ring homomorphism
$\pi^*:\Z[M_n]\to A$. By duality, $\pi$ is faithful iff $\pi^*$ is an epimorphism in $\mf R$.
In general an epimorphism in $\mf R$ is not surjective, e.\,g. a
localization homomorphism is an epimorphism. To get around this problem we need the following
fact. It will be used in Lemma~\ref{splitY} to prove the splitting property of a certain ideal.

\begin{lem}\label{EpiAndKer}
Let
$R_1 \stackrel\ph\to R_2 \stackrel\psi\to R$
be a sequence of morphisms in $\mf R$.
Suppose that $\ph$ is an epimorphism in $\mf R$ and $\psi\circ\ph$ is surjective. Then the kernel
of $\psi$ is generated by $\ph(\Ker\psi\circ\ph)$ as an ideal of $R_2$.
\end{lem}

\begin{proof}
Denote by $\mf a$ the kernel of $\psi\circ\ph$ and let $\mf b$ be the ideal of $R_2$
generated by $\ph(\mf a)$. Tensoring given sequence of morphisms with $R_1/\mf a$
over $R_1$ we obtain the commutative diagram
$$
\begin{CD}
R_1     @>\ph >>       R_2          @>\psi>> R\\
@V\rho_{\mf a}VV @V\rho_{\mf b}VV        @| \\
R_1/\mf a @>\ph'>>    R_2/\mf b     @>\psi'>> R
\end{CD}
$$
Since the left part of the diagram is a pushout square and $\ph$ is an epimorphism, $\ph'$
is an epimorphism as well. Surjectivity of $\psi\circ\ph$ implies that $\psi'\circ\ph'$
is an isomorphism. Denote by $\theta$ its inverse.
Notice that two maps $\ph'$ and $\ph'\circ\theta\circ\psi'\circ\ph'$ from $R_1/\mf a$ to
$R_2/\mf b$ are equal. Since $\ph'$ is an epimorphism, $\ph'\circ\theta\circ\psi'$ is the
identity map, hence $\psi'$ is injective. It follows that the kernel of $\psi$ is
equal to $\mf b$ as required.
\end{proof}

\begin{defn}\label{KeyConstr}
Let $G$ be a group scheme over $\Z$ and $\pi:G\to M_n$ a faithful representation
(by Lemma~\ref{faithful} it always exists).
Consider a homomorphism of polynomial rings
$\ph:\Z[M_n]\to \Z[t,y_{ij}\mid 1\le i,j\le n]$ sending $z_{ij}$ to $\delta_{ij}+ty_{ij}$
for all $1\le i,j\le n$. Define $D$ by the pushout diagram
$$
\begin{CD}
\Z[M_n] @>\ph >> \Z[t,y_{ij}]\\
@V{\pi^*}VV          @V\xi VV    \\
A       @>f>>           D
\end{CD}
$$
i.\,e. $D$ is the tensor product of $\Z[t,y_{ij}]$ and $A$ over $\Z[M_n]$.
By abuse of notation we denote the images of $t$ and $y_{ij}$'s in $D$ by the same
symbols. Define $f\in G(D)$ as the bottom horizontal arrow of the pushout diagram above.
The matrix entries of $\pi(f)$ are the images of $z_{ij}$'s, therefore
$\pi(f)=e+ty$, where $y$ is a matrix with entries $y_{ij}$.
In particular, $f$ vanishes modulo $t$, i.\,e. $f\in G(D,tD)$.
\end{defn}

Note that $\xi$ is an epimorphism in $\mf R$ as $\pi^*$ has this property.

\begin{prop}\label{versal}
The triple $(D,t,f)$ constructed above satisfies property~$\ref{univ}$.
Moreover, let $\mf a$ be an ideal of a ring $R$, $r\in R$, and $h\in G(R,r\mf a)$.
Then there exists a homomorphism $\theta:D\to R$ such that
$\theta(t)=r$, $\theta(f)=h$, and $\theta(y_{ij})\in\mf a$ for all $i,j=1,\dots,n$.
\end{prop}

\begin{proof}
By definition, $h$ is a homomorphism $A\to R$ and
$h\bigl(\pi^*(z_{ij})\bigr)$ are the matrix entries of $\pi(h)$.
Since $h\in G(R,r\mf a)$, we have $h\bigl(\pi^*(z_{ij}-\delta_{ij})\bigr)=r\tilde h_{ij}$
for some $\tilde h_{ij}\in\mf a$.
Define the homomorphism $\psi:\Z[t,y_{ij}]\to R$ by
$\psi(t)=r$ and $\psi(y_{ij})=\tilde h_{ij}$, so that
$\psi\circ\ph=h\circ\pi^*$. By definition of $D$ there exists a
unique homomorphism $\theta:D\to R$ such that $\theta\circ f=h$ and $\theta\circ\xi=\psi$.
The former equation shows that $\theta(f)=h$, whereas the latter implies that $\theta(t)=r$
and $\theta(y_{ij})=\tilde h_{ij}\in\mf a$.
\end{proof}

Denote by $Y$ the ideal in $D$ generated by $y_{ij}$'s. It plays the same role as the fundamental
ideal in the affine algebra of the group in the proof of the standard commutator formulas.
Here we prove that $Y$ is a splitting ideal.

\begin{lem}\label{splitY}
The ideal $Y$ is a splitting ideal.
Let $R$ be a ring and $s\in R$. Denote by $Y'$ the ideal $(Y\otimes R)/(t-s)$
of $(D\otimes R)/(t-s)$. Then $Y'$ is a splitting ideal.
Moreover, $(D\otimes R)/(t-s)=R\oplus Y'$ as additive groups.
\end{lem}

\begin{proof}
Consider the diagram
$$
\begin{CD}
\Z[M_n] @>\ph >> \Z[t,y_{ij}]\\
@V{\pi^*}VV      @V\rho VV    \\
A       @>e>>    \Z[t]
\end{CD}
$$
where $e$ is the identity element of $G(\Z[t])$ and $\rho$ is the reduction homomorphism modulo
the ideal, generated by $y_{ij}$'s. The element $e\bigl(\pi^*(z_{ij})\bigr)$
is the matrix entry of the identity matrix. Hence it is equal to
$\rho_Y\bigl(\ph(z_{ij})\bigr)=\rho_Y(\delta_{ij}+ty_{ij})=\delta_{ij}$.
It follows that the diagram commutes; therefore there exists a unique map
$\rho':D\to\Z[t]$ such that $\rho'\circ\xi=\rho$ and $\rho'\circ e=f$.
Thus, we have the following commutative diagram
$$
\begin{CD}\Z[t]\ \ & \hlrarrow & \ \ \Z[t,y_{ij}]  @>\xi>> D             @>\rho'>> \Z[t]\\
       @V{\otimes R}VV             @V{\otimes R}VV @VV{\otimes R}V       @VV{\otimes R}V\\
           R[t]\ \ & \hlrarrow & \ \  R[t,y_{ij}]  @>>> D\otimes R       @>>>       R[t]\\
       @V{\rho_{(t-s)}}VV       @V{\rho_{(t-s)}}VV @VV{\rho_{(t-s)}}V    @V{\rho_{(t-s)}}VV\\
           R  \ \  & \hlrarrow & \ \  R[y_{ij}]    @>\xi'>> D\otimes R/(t-s) @>\rho''>>       R   \\
\end{CD}
$$
where the second row is obtained from the first one by tensoring with $R$,
and the third row is the image of the second one under the reduction homomorphism modulo
$(t-s)$. To factor out the ideal $(t-s)$ is the same as tensoring with $R[t]/(t-s)$
over $R[t]$. Therefore all squares in the diagram are pushouts.
Since $\xi$ is an epimorphism in $\mf R$, the map $\xi'$ has the same property.
Clearly, the kernels of $\rho'\circ\xi$ and $\rho''\circ\xi'$ are generated by $y_{ij}$'s.
Now, by Lemma~\ref{EpiAndKer} $\Ker\rho'=Y$ and $\Ker\rho''=Y'$.
Thus, $D=\Z[t]\oplus Y$ and $D\otimes R/(t-s)=R\oplus Y'$.
\end{proof}

\begin{cor}\label{KeySplitting}
Let $\mf a$ be an ideal of $R$. With the notation of the previous lemma we have
\begin{multline*}
E\bigl(D\otimes R/(t-s),D\otimes\mf a/(t-s)\bigr)\cap G(D\otimes R/(t-s),Y')=\\
\EE\bigl(D\otimes R/(t-s),D\otimes\mf a/(t-s),Y'\bigr).
\end{multline*}
\end{cor}

\begin{proof}
The corollary follows immediately from the previous lemma and splitting principle~\ref{splitting}.
\end{proof}
\section{Extended dilation principle and key lemma}

In the current section $G$ is assumed to be simply connected.
We prove the Key Lemma of the current article which is a stronger version of Bak's key lemma.
The original statement obtained by A.\,Bak in~\cite{BakNonabelian} for $G=\GL_n$
is a special case of Key Lemma~\ref{key} with $R'=\q=R$ and $\ph=\operatorname{id}$.
Later Bak's key lemma was extended to all Chevalley groups by R.\,Hazrat and N.\,Vavilov
in~\cite{HazVav}. The relative version appeared in~\cite{BHVstrike} by the three above-mentioned
authors. In these articles it was proved using
double localization -- the method that required commutator calculus called
``yoga of commutators''. The lemma is a key computational step in the existing
proofs of the nilpotent structure of $K_1$ and the boundedness of commutator width.

If we replace $G(R,s^mR)$ by $E(R,s^mR)$, then the usual dilation principle works
(see~\cite{StepComput}).
Namely, it suffices to consider a generator of $E(s^mR)$ and,
since $\Z[X_\alpha]=\Z[t]$, we have a generic element
$x_\alpha(t)\in E(R[t],tR[t])$ for such a generator.
Afterwards we specialize the independent variable $t$ to $s^m$ to get the result.

For the general case we constructed a generic element of $G(R,s^mR)$ in the previous section.
But now $t$ is not an independent variable and we cannot apply the usual dilation principle.
To get around this problem we prove an extension of the principle.
We shall see that over localized ring the principle follows almost immediately from
Lemma~\ref{EE}. The main concern of the proof is to pull the result back to the group over
original ring. This is done in the following lemma which uses the method by A.\,Bak and the author
of struggling against zero divisors.

\begin{lem}\label{zerodiv}
Let $\q$ be an ideal of a ring $R$, $s\in R$, and $a\in G(R,s\q)$.
If $\lambda_s(a)\in\lambda_s\bigl(E(R,s\q)\bigr)$, then $a\in E(R,s\q)$.
\end{lem}

\begin{proof}
Let $N$ denote the nilpotent radical of $R$, $\bar s=\rho_N(s)$, $\bar\q=\rho_N(\q)$, and
$\bar a=\rho_N(a)\in G(R/N,\bar s\bar\q)$. Since $R_s/N_s$ is naturally
isomorphic to $(R/N)_{\bar s}$, we have
$\lambda_{\bar s}(\bar a)\in\lambda_{\bar s}\bigl(E(R/N,\bar s\bar\q)\bigr)$.
Since $R/N$ is a reduced ring, by~\cite[Lemma~5.5]{BakStepNet} the restriction of
$\lambda_{\bar s}$ on $\bar sR/N$ is injective. Therefore $\bar a\in E(R/N,\bar s\bar\q)$.
Since $E$ preserves surjective homomorphisms, there exists $b\in E(R,s\q)$ such that
$\rho_N(b)=\bar a$. Now, $a\in bG(R,N\cap s\q)$ and $G(R,N\cap s\q)$ is contained
in $E(R,s\q)$ by Lemma~\ref{radical}.
\end{proof}

\begin{lem}\label{exdilation}
Let $R$ be a ring, $s,t\in R$ and $a\in G(R,tR)$. Suppose that $\lambda_s(a)\in E(R_s,tR_s)$.
Then there exists $m\in\N$ such that $\rho(a)\in E\bigl(R/(t-s^m)\bigr)$, where
$\rho:R\to R/(t-s^m)$ is the reduction homomorphism.
\end{lem}

\begin{proof}
Let $R'=R[u]/(t-u^3)$, $\iota:R\to R'$ be the canonical inclusion, and $a'=\iota(a)\in G(R',tR')$.
Clearly, $\lambda_s(a')\in E(R'_s,u^3R'_s)$. By Lemma~\ref{EE} $E(R'_s,u^3R'_s)\le E(uR'_s)$,
hence $\lambda_s(a')=\prod_ix_{\alpha_i}(\frac{ur_i}{s^n})$ for some $\alpha_i\in\Phi$ and
$r_i\in R$. Put $m=3(n+1)$. Define the homomorphism $\rho$ to be the composition of $\iota$
with the reduction homomorphism modulo $(u-s^{n+1})R'$. Clearly,
$$
R'/(u-s^{n+1})=R[u]/(t-u^3,u-s^{n+1})=R[u]/(u-s^{n+1},t-s^{m})=R/(t-s^m)
$$
and $\rho(t)=\rho(s^m)$. It is easy to see that $\rho$ is
the reduction homomorphism modulo $(t-s^m)$. Put $s'=\rho(s)$
and $r'_i=\rho(r_i)$.
Since a localization homomorphism commutes with a reduction homomorphism,
we have
$$
\lambda_{s'}\Bigl(\rho(a)\Bigr)=
\prod_ix_{\alpha_i}(s'r'_i)\in\lambda_{s'}\Bigl(E\bigl(s'R/(t-s^m)\bigr)\Bigr).
$$
On the other hand, $\rho(a)\in G\bigl(R/(t-s^m),s'R/(t-s^m)\bigr)$.
Thus, by Lemma~\ref{zerodiv} we have $\rho(a)\in E\bigl(R/(t-s^m)\bigr)$.
\end{proof}

Now we are at the position to prove the key technical statement. 

\begin{keylem}\label{key}
Let $R$ be a ring, $a\in G(R)$, and $s\in R$. Suppose that $\lambda_s(a)\in E(R_s)$.
Then there exists $m\in\N$ such that for any ring $R'$, homomorphism $\ph:R\to R'$, and
ideal $\q$ of $R'$ we have $[\ph(a),G(R',\ph(s)^m\q)]\le E(R',\q)$.
\end{keylem}

\begin{proof}
First, we consider the commutator $[a,f]\in G(D\otimes R)$, where $D$ is the ring,
constructed in the previous section.
Rings $R$ and $D$ will be identified with their canonical images
in $D\otimes R$.
By standard commutator formula~\ref{stcommform}
$\lambda_s\bigl([a,f]\bigr)\in E\bigl(D\otimes R_s,tD\otimes R_s\bigr)$.
By Lemma~\ref{exdilation} there exists $m\in\N$ such that
$\rho\bigl([a,f]\bigr)\in E\bigl(D\otimes R/(t-s^m)\bigr)$, where $\rho=\rho_{(t-s^m)}$
is the reduction homomorphism.
Let $Y'$ be the ideal of $D\otimes R/(t-s^m)$, generated by the images of $y_{ij}$'s.
Since, $f$ vanishes modulo $Y$, we have
$\rho\bigl([a,f]\bigr)\in G\bigl(D\otimes R/(t-s^m),Y'\bigr)$.
By Corollary~\ref{KeySplitting}
we have $\rho\bigl([a,f]\bigr)\in E\bigl(D\otimes R/(t-s^m),Y'\bigr)$,
hence $[a,f]\in E\bigl(D\otimes R,Y\bigr)G\bigl(D\otimes R,(t-s^m)\bigr)$.

Now, let $b\in G(R',\ph(s)^m\q)$. By Proposition~\ref{versal} there exists a homomorphism
$\theta:D\to R'$ such that $\theta(t)=\ph(s)^m$, $\theta(f)=b$, and $\theta(Y)\subseteq\q$.
The homomorphism $\theta\tensor\ph:D\otimes R\to R'$ maps $t$ and $s^m$ to the same element
of $R'$, hence
$\theta\tensor\ph\Bigl(G\bigl(D\otimes R,(t-s^m)\bigr)\Bigr)=\{e\}$.
Thus,
$$
[\ph(a),b]=\theta\tensor\ph\bigl([a,f]\bigr)\in
\theta\tensor\ph\bigl(E(D\otimes R,Y)\bigr)\le E(R',\q).
$$
\end{proof}

\section{Extended elementary group}

In this section we construct an extended (relative) elementary group and
a generic element for this group.
Let $l$ be a natural number. Denote by $\Um_l(R)$ the set of all unimodular sequences
over $R$ of length $l$.

\begin{defn}\label{defEtilde}
Let $\mf a$ be an ideal of a ring $R$.
Define an extended relative elementary group $\wt E(R,\mf a)$ by the formula
$$
\wt E(R,\mf a)=\wt E^{(l)}(R,\mf a)=\bigcap_{(r_1,\dots,r_l)\in\Um_l(R)}\left(\prod_{k=1}^l G(R,r_k\mf a)\right).
$$
\end{defn}

The definition of $\wt E$ depends on $l$. In the sequel for simply connected groups we put $l=\#W$
as in Corollary~\ref{sc-loc}. On the other hand, the statements of the current section hold independently
of $l$ and of the type of the group.

\begin{lem}\label{Etilde}
$E(R,\mf a)\le\wt E(R,\mf a)$.
\end{lem}

\begin{proof}
Clearly, $\wt E(R,\mf a)$ is normal in $G(R)$. Therefore, it suffices to show
that it contains $x_\alpha(q)$ for all $\alpha\in\Phi$ and $q\in\mf a$.
If $\sum_{k=1}^lp_kr_k=1$, then
$x_\alpha(q)=\prod_{k=1}^l x_\alpha(p_kr_kq)\in\prod_{k=1}^l G(R,r_k\mf a)$, which
implies the result.
\end{proof}

Let $D^{(k)}$ denote the ring isomorphic to $D$ generated by $t^{(k)}$ and $y_{ij}^{(k)}$
instead of $t$ and $y_{ij}$ (where $1\le i,j\le n$).
Similarly, we write $f^{(k)}\in G(D^{(k)})$ instead of $f\in G(D)$.
Let
$$
U=\bigotimes_{k=1}^lD^{(k)}.
$$
We identify the elements of each $D^{(k)}$ with their canonical images in $U$.
Similarly, the group elements of $G(D^{(k)})$ are identified with their images
in $G(U)$. Denote by $Y^{(k)}$ the ideal of $U$ generated by $y_{ij}^{(k)}$
for all $1\le i,j\le n$ and let $\wt Y=\sum_{k=1}^l Y^{(k)}$.
We show that the element $u=f^{(1)}\cdots f^{(l)}\in G(U)$ is a generic
element of $\wt E$ in the following sense.

\begin{lem}\label{tildeversal}
Given a ring $R$, a unimodular sequence $r_1,\dots,r_l\in R$, and $b\in\wt E(R,\mf q)$,
there exists a homomorphism $\eta:U\to R$ such that $\eta(u)=b$,
$\eta(t^{(k)})=r_k$ for all $k=1,\dots,l$ and $\eta(\wt Y)\subseteq\mf q$.
\end{lem}

\begin{proof}
By definition of $\wt E(R,\mf q)$ we can write $b$ as a product
of $b^{(k)}\in G(R,r_k\mf q)$ as $k$ ranges from $1$ up to $l$.
By Property~\ref{univ} for each $k=1,\dots,l$ there exists a homomorphism
$\theta_k:D^{(k)}\to R$ sending $t^{(k)}$ to $r_k$ and $f^{(k)}$ to $b^{(k)}$.
Moreover, by Proposition~\ref{versal} $\theta_k(y_{ij}^{(k)})\in\mf q$.
By the universal property of tensor product there exists a unique homomorphism
$\eta:U\to R$ such that each $\theta_k$ factors through $\eta$.
It follows that $\eta$ satisfies the conditions of the lemma.
\end{proof}

\section{Width of commutators}

Second important property of the extended elementary group is the
commutator formula from the following theorem.
We prove it together with the simplest applications of our constructions
for the width of commutators.
Let $\Sigma$ be a functorial generating set for the relative elementary subgroup
(it is defined before corollary~\ref{length0}).

\begin{thm}\label{CommForm}
Let $\q$ be an ideal of a ring $R$. Suppose that $G$ is simply connected. Then
$$
[G(R),\wt E(R,\q)]\le E(R,\q).
$$

Moreover, there exists a constant $L\in\N$ such that for any ring $R$,
ideal $\q$ of $R$, $a\in G(R)$, and $b\in \wt E(R,\q)$ the width of the
commutator $[a,b]$ with respect to $\Sigma(R,\q)$ is at most $L$.
\end{thm}

\begin{proof}
By Lemma~\ref{sc-loc} there exists a unimodular sequence $s_1,\dots,s_l\in A$
such that $\lambda_{s_k}(g)\in E(A_{s_k})$. For each $k=1,\dots,l$ take $m_k\in\N$
satisfying conditions of Key Lemma~\ref{key} with $R=A$, $a=g$, and $s=s_k$.
Denote by $T$ the ideal of $U\otimes A$ generated by $t^{(k)}-s_k^{m_k}$ for
all $k=1,\dots,l$ (as usual, we identify elements of $U$ and $A$ with their
images in $U\otimes A$).
Denote by $R'$ the quotient ring $U\otimes A/T$.
Let $\ph:A\to R'$ and $\psi:U\to R'$ be the canonical homomorphisms.
Note that $\psi\tensor\ph$ is the reduction homomorphism modulo $T$,
in particular it is surjective.
Put $h^{(k)}=\psi(f^{(k)})$ and $Y'=\psi\tensor\ph(\wt Y\otimes A)$.
Since $f^{(k)}\in G(D^{(k)},t^{(k)}Y^{(k)})$ and $\psi(t^{(k)})=\ph(s_k)^{m_k}$,
we have $h^{(k)}\in G(R',\ph(s_k)^{m_k}Y')$. By the choice of $m_k$ we have
$[\ph(g),h^{(k)}]\in E(R',Y')$. Since $E(R',Y')$ is normal in $G(R')$, commutator
identity~\ref{CommId} implies that $[\ph(g),\psi(u)]\in E(R',Y')$.
Therefore, $[g,u]\in vG(U\otimes A,T)$ for some $v\in E(U\otimes A,\wt Y\otimes A)$.
Denote by $L$ the width of $v$ with respect to $\Sigma(U\otimes A,\wt Y\otimes A)$.

Now, we construct a homomorphism $\pi:U\otimes A\to R$ that
maps $g$ to $a$, $u$ to $b$, $\wt Y\otimes A$ to $\q$, and $T$ to zero.
Recall, that $a$ itself is a ring homomorphism $A\to R$ such that $a(g)=a$.
For all $k=1,\dots,l$ denote $r_k=a(s_k^{m_k})$. Since $s_1,\dots,s_l$ is a unimodular
sequence, $s_1^{m_1},\dots,s_l^{m_l}$ is also unimodular (otherwise they were contained in
a maximal ideal, but then $s_1,\dots,s_l$ were contained in the same ideal, a contradiction).
Hence, its image $r_1,\dots,r_l$ is unimodular as well.
By Lemma~\ref{tildeversal} there exists a homomorphism $\eta:U\to R$
such that $\eta(u)=b$,
$\eta(t^{(k)})=r_k$ for all $k=1,\dots,l$ and $\eta(\wt Y)\subseteq\mf q$.
Put $\pi=\eta\tensor a$. Then $\pi(g)=a(g)=a$, $\pi(u)=\eta(u)=b$,
$\pi(\wt Y\otimes A)=\eta(\wt Y)\subseteq\q$, and
$\pi(t^{(k)}-s_k^{m_k})=\eta(t^{(k)})-a(s_k^{m_k})=0$ as required.
Finally,
$$
[a,b]=\pi\bigl([g,u]\bigr)=\pi\bigl(v\bigr)\in\pi\bigl(E(U\otimes A,\wt Y\otimes A)\bigr)\le
E(R,\q)
$$
and the width of $[a,b]$ with respect to $\Sigma(R,\q)$ is bounded by the width of $v$ with respect to
$\Sigma(U\otimes A,\wt Y\otimes A)$.
\end{proof}

The following statement is an immediate corollary of the theorem.

\begin{cor}\label{commlength}
Let $G$ be a simply connected Chevalley group.
There exists an integer $L$ depending only on $G$ satisfying the following property:
Given $a\in G(R)$ and $b\in E(R)$ the commutator $[a,b]$ can be written as a product of
at most $L$ elementary root unipotent elements.
\end{cor}

\section{Relative Key Lemma}

In this section we prepare technical tools for relative commutator formula~\ref{RelThm}
by proving a relative analog of Lemma~\ref{key}.
Let $\q$ be an ideal of  a ring $R$ and $t\in R$.
Denote by $E(tR,t\q)$ the normal closure of $E(t\q)$ in $E(tR)$.
The following result is a consequence of~Lemma~7.1 of~\cite{StepComput}.

\begin{lem}\label{relEE}
$\EE(R,\q,t^{27}R)\le E(t\q,tR)$.
\end{lem}

\begin{proof}
Let $x$ be an independent variable. By definition
$$
\EE(R[x],\q[x],x^{27}R)=E(R[x],x^{27}\q[x])\cdot[E(R[x],\q[x]),E(R[x],x^{27}R[x])].
$$
The first part of the proof of~\cite[Lemma~7.1]{StepComput} shows that
the second factor is contained in $E(R[x],x^3\q[x])$. Clearly, the first factor lies in
the same group. In the second part of the proof of~\cite[Lemma~7.1]{StepComput}
the inclusion $E(R[x],x^3\q[x])\le E(xR[x],x\q[x])$ is obtained.
Thus, $\EE(R[x],\q[x],x^{27}R[x])\le E(xR[x],x\q[x])$.

The evaluation map $R[x]\to R$ sending $x$ to $t$ is surjective and the functor $\EE$
preserves surjective maps. Applying the evaluation map to the above inclusion we get the result.
\end{proof}

The following statement is a suitable generalization of the relative dilation principle
proved in~\cite{ApteChattRao}, \cite{ApteStep}, and~\cite{StepComput}.

\begin{lem}\label{ExtRelDilation}
Let $\q$ be an ideal of  a ring $R$, $s,t\in R$, and $a\in G(R,tR)$.
Suppose that $\lambda_s(a)\in\EE(R_s,tR_s,\q)$.
Then there exists $m\in\N$ such that $\rho(a)\in E\bigl(R/(t-s^m),\rho(\q)\bigr)$, where
$\rho=\rho_{(t-s^m)}:R\to R/(t-s^m)$ is the reduction homomorphism.
\end{lem}

\begin{proof}
The proof is essentially the same as for Lemma~\ref{exdilation} using Lemma~\ref{relEE}
instead of Lemma~\ref{EE}.
\end{proof}

Now we prove the relative commutator formula which is a common generalization of
commutator formulas~\ref{stcommform}. By different methods and in different settings
it was obtained in~\cite{VavStepCommForm0}, \cite{VavStepCommForm},
and~\cite{HazZhangCommForm}. Since we do not want to exclude the cases where the elementary
group is not perfect, we give yet another proof based on splitting principle~\ref{splitting}.

\begin{lem}\label{RelCommForm}
Let $\mf a$ and $\mf b$ be ideals of a ring $R$. Then
$$
[E(R,\mf a),G(R,\mf b)]\le\EE(R,\mf a,\mf b).
$$
\end{lem}

\begin{proof}
Consider the ring $R\otimes A$. As before, we identify elements of $R$, $A$, $G(R)$, and $G(A)$
with their canonical images in $R\otimes A$ and $G(R\otimes A)$.
By the first standard commutator formula of Lemma~\ref{stcommform} we have
$[E(R,\mf a),g]\le E(R\otimes A,\mf a\otimes A)$. On the other hand, since $g\in G(A,I)$, we have
$[E(R,\mf a),g]\le G(R\otimes A,R\otimes I)$. Since $R\otimes I$ is a splitting ideal of
$R\otimes A$, by Lemma~\ref{splitting} we have $[E(R,\mf a),g]\le\EE(R\otimes A,\mf a\otimes A,R\otimes I)$.

Now, let $b\in G(R,\mf b)$. Applying $\operatorname{id}_R\tensor b$ to the inclusion above we obtain
$[E(R,\mf a),b]\le\EE(R,\mf a,\mf b)$ which completes the proof.
\end{proof}

The following statement generalizes Key Lemma~\ref{key}.

\begin{lem}\label{relkey}
Let $R$ be a ring, $a\in G(R,\mf a)$, and $s\in R$. Suppose that $\lambda_s(a)\in E(R_s,\mf a_s)$.
Then there exists $m\in\N$ such that for any ring $R'$, homomorphism $\ph:R\to R'$, and ideals
$\mf a'\supseteq\ph(\mf a)$ and $\mf b'$ of $R'$ we have
$[\ph(a),G(R',\ph(s)^m\mf b')]\le\EE(R',\mf a',\mf b')$.
\end{lem}

\begin{proof}
First, we consider the commutator $[a,f]\in G(D\otimes R)$, where $D$ is the ring,
constructed in section~\ref{KeyConstruction}. The rings $R$ and $D$ will be identified with
there canonical images in $D\otimes R$. Since $f\in tD\otimes R$,
by the relative commutator formula~\ref{RelCommForm} we have
$\lambda_s\bigl([a,f]\bigr)\in\EE\bigl(D\otimes R_s,tD\otimes R_s,D\otimes\mf a\bigr)$.
By Lemma~\ref{ExtRelDilation} there exists $m\in\N$ such that
$\rho\bigl([a,f]\bigr)\in E\bigl(D\otimes R/(t-s^m),\rho(D\otimes \mf a)\bigr)$, where
$\rho=\rho_{(t-s^m)}$ is the reduction homomorphism.
Let $Y'$ be the ideal of $D\otimes R/(t-s^m)$, generated by the images of $y_{ij}$'s.
Since, $f$ vanishes modulo $Y$, we have
$\rho\bigl([a,f]\bigr)\in G\bigl(D\otimes R/(t-s^m),Y'\bigr)$.
By Corollary~\ref{KeySplitting}
we have $\rho\bigl([a,f]\bigr)\in\EE\bigl(D\otimes R/(t-s^m),\rho(D\otimes \mf a),Y'\bigr)$,
hence $[a,f]\in\EE\bigl(D\otimes R,D\otimes \mf a,Y\bigr))G\bigl(D\otimes R,(t-s^m)\bigr)$.

Now, let $b\in G(R',\ph(s)^m\mf b')$. By Proposition~\ref{versal} there exists a homomorphism
$\theta:D\to R'$ such that $\theta(t)=\ph(s)^m$, $\theta(f)=b$, and $\theta(Y)\subseteq\mf b'$.
The homomorphism $\theta\tensor\ph:D\otimes R\to R'$ maps $t$ and $s^m$ to the same element
of $R'$, hence $\theta\tensor\ph\Bigl(G\bigl(D\otimes R,(t-s^m)\bigr)\Bigr)=\{e\}$.
Thus,
$$
[\ph(a),b]=\theta\tensor\ph\bigl([a,f]\bigr)\in
\theta\tensor\ph\bigl(\EE(D\otimes R,D\otimes \mf a,Y)\bigr)\le\EE(R',\mf a',\mf b').
$$
\end{proof}

\section{Relative commutator formulas}

In this section for a simply connected Chevalley group we prove a stronger version of
relative commutator formula~\ref{RelCommForm} together with a bound for the width of
commutators. As a consequence we prove the multi-commutator formula from~\cite{HazZhangMulti}.

First, we define a functorial generating set for the group functor $\EE$.
Clearly, $\EE$ is a functor on the category of pairs of ideals which is defined as follows.
An object is a triple $(R,\mf a,\mf b)$, where $\mf a$ and $\mf b$ are ideals of a ring $R$.
A morphism $\ph:(R,\mf a,\mf b)\to(R',\mf a',\mf b')$ in this category is a ring
homomorphism $\ph:R\to R'$ such that $\ph(\mf a)\subseteq\mf a')$ and
$\ph(\mf b)\subseteq\mf b')$. Let $\Sigma$ be a functor from the category of pairs of ideals
to the category of sets. We call $\Sigma$ a \emph{functorial generating set} for $\EE$ if
there is a natural inclusion of $\Sigma(R,\mf a,\mf b)$ to the underlying
set of $\EE(R,\mf a,\mf b)$ and $\Sigma(R,\mf a,\mf b)$ spans $\EE(R,\mf a,\mf b)$
for any object $(R,\mf a,\mf b)$.
A natural functorial generating set for $[E(R,\mf a),E(R,\mf b)]$ is obtained in~\cite[Theorem~3]{HVZgen}
under additional assumptions, which are necessary exactly for the equality
$\EE(R,\mf a,\mf b)=[E(R,\mf a),E(R,\mf b)].$
Actually, the proof of~\cite[Theorem~3]{HVZgen} works for $\EE(R,\mf a,\mf b)$ instead of
$[E(R,\mf a),E(R,\mf b)]$ and for $\Phi\ne A_1$ without any conditions.

\begin{thm}\label{RelThm}
Let $\mf a$ and $\mf b$ be ideals of a ring $R$. Suppose that $G$ is simply connected. Then
$$
[\wt E(R,\mf a),G(R,\mf b)]\le\EE(R,\mf a,\mf b).
$$

Moreover, given a functorial generating set $\Sigma$ for $\EE$ there exist a constant $L\in\N$
such that for any ring $R$,
ideals $\mf a$ and $\mf b$ of $R$, $a\in G(R,\mf b)$, and $b\in \wt E(R,\mf a)$
the width of the commutator $[a,b]$ with respect to $\Sigma(R,\mf a,\mf b)$ is at most $L$.
\end{thm}

\begin{proof}
By Lemma~\ref{sc-loc} there exists a unimodular sequence $s_1,\dots,s_l\in A$
such that $\lambda_{s_k}(g)\in E(A_{s_k},I_{s_k})$. For each $k=1,\dots,l$ take $m_k\in\N$
satisfying the conditions of Key Lemma~\ref{relkey} with $R=A$, $a=g$, and $s=s_k$.
Denote by $T$ the ideal of $U\otimes A$ generated by $t^{(k)}-s_k^{m_k}$ for
all $k=1,\dots,l$ (as usual, we identify elements of $U$ and $A$ with their
images in $U\otimes A$).
Denote by $R'$ the quotient ring $U\otimes A/T$.
Let $\ph:A\to R'$ and $\psi:U\to R'$ be the canonical homomorphisms.
Note that $\psi\tensor\ph$ is the reduction homomorphism modulo $T$,
in particular it is surjective.
Put $h^{(k)}=\psi(f^{(k)})$ and $Y'=\psi\tensor\ph(\wt Y\otimes A)$
and $I'=\psi\tensor\ph(U\otimes I)$.
Since $f^{(k)}\in G(D^{(k)},t^{(k)}Y^{(k)})$ and $\psi(t^{(k)})=\ph(s_k)^{m_k}$,
we have $h^{(k)}\in G(R',\ph(s_k)^{m_k}Y')$. By the choice of $m_k$ we have
$[\ph(g),h^{(k)}]\in\EE(R',Y',I')$. Since $\EE(R',Y',I')$ is normal in $G(R')$, commutator
identity~\ref{CommId} implies that $[\ph(g),\psi(u)]\in E(R',Y')$.
Therefore, $[g,u]\in vG(U\otimes A,T)$ for some $v\in\EE(U\otimes A,\wt Y\otimes A,U\otimes I)$.
Denote by $L$ the width of $v$ with respect to $\Sigma(U\otimes A,\wt Y\otimes A,U\otimes I)$.

Now, we construct a homomorphism $\pi:U\otimes A\to R$ that
maps $g$ to $a$, $u$ to $b$, $\wt Y\otimes A$ into $\mf a$, $U\otimes I$ into $\mf b$,
and $T$ to zero. Recall, that $a$ itself is a ring homomorphism $A\to R$ such that $a(g)=a$
and $a(I)\subseteq\mf b$.
For all $k=1,\dots,l$ denote $r_k=a(s_k^{m_k})$. Since $s_1,\dots,s_l$ is a unimodular
sequence, the sequence $r_1,\dots,r_l$ is unimodular as well.
By Lemma~\ref{tildeversal} there exists a homomorphism $\eta:U\to R$
such that $\eta(u)=b$,
$\eta(t^{(k)})=r_k$ for all $k=1,\dots,l$ and $\eta(\wt Y)\subseteq\mf q$.
Put $\pi=\eta\tensor a$. Then $\pi(g)=a(g)=a$, $\pi(u)=\eta(u)=b$,
$\pi(\wt Y\otimes A)=\eta(\wt Y)\subseteq\mf a$, $\pi(U\otimes I)=a(I)\subseteq\mf b$
and $\pi(t^{(k)}-s_k^{m_k})=\eta(t^{(k)})-a(s_k^{m_k})=0$ as required.
Finally,
$$
[a,b]=\pi\bigl([g,u]\bigr)=\pi\bigl(v\bigr)\in
\pi\bigl(\EE(U\otimes A,\wt Y\otimes A,U\otimes I)\bigr)\le E(R,\mf a,\mf b)
$$
and the width of $[a,b]$ with respect to $\Sigma(R,\mf a,\mf b)$ is bounded by the width of
$v$ with respect to $\Sigma(U\otimes A,\wt Y\otimes A,U\otimes I)$.
\end{proof}

Next, we observe a handy property of $\wt E$. Here we do not require $G$ to be simply connected.

\begin{lem}\label{tildeEbig}
Let $\mf a$ and $\mf b$ be ideals of a ring $R$. Then
$[\wt E(R,\mf a),G(R,\mf b)]$ and $\EE(R,\mf a,\mf b)$ are contained
in $\wt E(R,\mf a\mf b)$.
\end{lem}

\begin{proof}
Let $r_1,\dots,r_l$ be a unimodular sequence. Using Lemma~\ref{GG} we have
$$
[\prod_{k=1}^lG(R,r_k\mf a),G(R,\mf b)]=\prod_{k=1}^l[G(R,r_k\mf a),G(R,\mf b)]\le
\prod_{k=1}^lG(R,r_k\mf a\mf b).
$$
This proves that $[\wt E(R,\mf a),G(R,\mf b)]\le\wt E(R,\mf a\mf b)$.
By Lemma~\ref{Etilde} $E(R,\mf a\mf b)\le\wt E(R,\mf a\mf b)$,
hence $\EE(R,\mf a,\mf b)\le\wt E(R,\mf a\mf b)$.
\end{proof}

The following multi-commutator formula for $G=\GL_n$ was obtained by R.Hazrat and Z.Zuhong
in~\cite{HazZhangMulti} (see Corollary~17 and its proof):
If $\mf a_1,\dots,\mf a_m$ are ideals of a ring $R$, then
\begin{multline*}
[E(R,\mf a_1),G(R,\mf a_2),\dots,G(R,\mf a_m)]=[E(R,\mf a_1),\dots,E(R,\mf a_m)]=\\
[E(R,\mf a_1\dots\mf a_{m-1}),E(R,\mf a_m)].
\end{multline*}
We prove even better result for any simply connected Chevalley group $G$ replacing
$E(R,\mf a_1)$ by $\wt E(R,\mf a_1)$ in the left hand side of the formula.
However, this stronger formula probably does not hold if the torus of $G$ does not belong
to the elementary group.

\begin{thm}\label{multi}
Suppose that $G$ is simply connected. Let $\mf a_1,\dots,\mf a_m$ be ideals of a ring $R$. Then
$$
[\wt E(R,\mf a_1),G(R,\mf a_2),\dots,G(R,\mf a_m)]\le\EE(R,\mf a_1\dots\mf a_{m-1},\mf a_m).
$$
\end{thm}

\begin{proof}
Applying Lemma~\ref{tildeEbig} $m-2$ times we get
$$
[\wt E(R,\mf a_1),G(R,\mf a_2),\dots,G(R,\mf a_m)]\le[\wt E(R,\mf a_1\dots\mf a_{m-1}),G(R,\mf a_m)].
$$
Now the result follows from Theorem~\ref{RelThm}.
\end{proof}

It follows from the proof of the previous theorem that the width of multi-commutators is
bounded by at most the same constant as the width of relative commutators in Theorem~\ref{RelThm}.

\section[Nilpotent structure of K1]{Nilpotent structure of $\K_1$}

In this section $G$ is assumed to be simply connected.
We prove a multi-relative version of Bak's theorem on the nilpotent structure of
$K_1$. Actually, we concentrate our attention on the induction step of the proof.
The base of induction
follows from bi-relative version of surjective stability which is available at the time only for
the special linear group by A.\,Mason and W.\,Stothers~\cite{MasonStothers}.
Therefore, we use an axiomatic notion
of dimension function istead of actual dimension of a ring. The set of axioms is similar to
those of dimension function of A.\,Bak~\cite{BakDimension,HazDim} for a particular kind of
infrastructure and structure squares.

\begin{defn}\label{DimFunc}
Let $\delta=\delta_G$ be a function from the class of rings to $\Z\cup\{\infty\}$ satisfying
the following properties.

\begin{enumerate}
\item
If $\delta(R)\le1$, then for any ideals $\mf a$ and $\mf b$ of $R$ we have
$[G(R,\mf a),G(R,\mf b)]\le\EE(R,\mf a,\mf b)$.
\item
Suppose that $0<\delta(R)<\infty$.
Given an ideal $\mf a$ of a ring $R$ there exists a multiplicative subset $S$ of $R$ such that
$G(S^{-1}R,S^{-1}\mf a)=E(S^{-1}R,S^{-1}\mf a)$ and $\delta(R/rR)<\delta(R)$ for all $r\in S$.
\item
If $R$ is not Noetherian, then $\delta(R)=\infty$.
\end{enumerate}

Then $\delta$ is called a dimension function for $G$.
\end{defn}

For example, the Bass--Serre dimension $\operatorname{BS-dim} R$ of a Noetherian ring $R$ is a dimension function
in this sense. To be more precise let us recall some definitions. A topological space is called
Noetherian if its closed subsets satisfy the descending chain condition. A space is irreducible if
it is not a union of two nonempty proper closed subsets.
The Krull (or combinatorial) dimension of a topological space is the supremum
of the lengths of proper chains of nonempty closed irreducible subsets.
The following definition was introduced by A.\,Bak in~\cite{BakNonabelian}.
Bass--Serre dimension $d=\operatorname{BS-dim} R$ of a ring $R$ is the smallest integer such that the
maximal spectrum $\Max R$ is a union of a finite number of irreducible subspaces of Krull dimension
not greater than $d$. If there is no such integer, then $\operatorname{BS-dim} R=\infty$. 
Clearly, $\operatorname{BS-dim} R\le\dim\Max R$.
An example where $\operatorname{BS-dim} R\ne\dim\Max R$ is given in~\cite[Ch.~3, Proposition~3.13]{BassBook}.
Put
$$
\delta_{\mathrm{BS}}(R)=\begin{cases}
\operatorname{BS-dim}R,&\text{ if }R\text{ is Noetherian}\\
\infty,&\text{ otherwise}\end{cases}
$$
To check that $\delta_{\mathrm{BS}}$ is a dimension function in our sense we prove the following statement.

\begin{lem}
Let $\mf a$ be an ideal of a semilocal ring $R$.
If $G$ is simply connected, then $G(R,\mf a)=E(R,\mf a)$.
\end{lem}

\begin{proof}
Let $a\in G(R,\mf a)$. Denote by $\mf q$ the Jacobson radical of $R$. Then
$R/\q=\bigoplus_{i=1}^n F_i$ for some fields $F_1,\dots,F_n$. Without loss of generality we may assume
that $\rho_\q(\mf a)=\bigoplus_{i=1}^m F_i$ for some $m\le n$.
Since $G$ is simply connected, $G(F)=E(F)$ for any field $F$.
Since $G$ and $E$ preserve direct sums,
\begin{multline*}
\rho_\q(a)\in G\bigl(R/\q,\rho_\q(\mf a)\bigr)=G(F_1)\times\dots\times G(F_m)=\\
E(F_1)\times\dots\times E(F_m)=E\bigl(R/\q,\rho_\q(\mf a)\bigr)
\end{multline*}
The functor $E$ preserves surjective maps, therefore there exists $a'\in E(R,\mf a)$ such that
$\rho_\q(a)=\rho_\q(a')$. It follows that $a\in a'b$ for some $b\in G(R,\q)$.
On the other hand, $b=a{a'}^{-1}\in G(R,\mf a)$. Thus, $b\in G(R,\q)\cap G(R,\mf a)=G(R,\mf a\cap\q)$.
By~\cite[Proposition~2.3]{AbeSuzuki} $G(R,\mf b)=E(R,\mf b)$ for any radical ideal $\mf b$,
which implies that $b\in E(R,\mf a\cap\q)$ and $a\in E(R,\mf a)$.
\end{proof}

\begin{cor}
The function $\delta_{\mathrm{BS}}$ is a dimension function in the sense of Definition~$\ref{DimFunc}$
for any simply connected Chevalley--Demazure group scheme $G$.

If $G=\SL_n$, then $\delta'_{\mathrm{BS}}(R)=\delta_{\mathrm{BS}}(R)-n+3$ already is a dimension
function for $G$.
\end{cor}

\begin{proof}
Property~\ref{DimFunc}(2) is verified in the proof of~\cite[Lemma~4.17]{BakNonabelian}
and~\ref{DimFunc}(3) holds by definition.
If $\operatorname{BS-dim}R=0$, then $\Max R$ is discrete and finite, i.\,e. $R$
is semilocal. Thus, property~\ref{DimFunc}(1) follows from the previous lemma
and Corollary~\ref{IndStep} below.

Clearly, if~\ref{DimFunc}(2) holds for $\delta$ then it holds also for $\delta'=\delta-k$, where $k$
is a given positive integer. The same is true for the third property of dimension function.
Let $G=\SL_n$. By~\cite[Ch.3, Theorem~3.5]{BassBook}
the stable rank $\operatorname{sr} R$ is bounded by $\operatorname{BS-dim}R+1$.
If $\delta'_{\mathrm{BS}}(R)\le1$, then $\operatorname{BS-dim}R\le n-2$ and $\operatorname{sr} R\le n-1$.
Now, property~\ref{DimFunc}(1) for $\delta'_{\mathrm{BS}}$ follows
from~\cite[Theorem~1.2]{MasonCommBis} by A.\,Mason and W.\,Stothers, see
also~\cite[Corollary~3.9]{MasonStothers}.
\end{proof}

It is expected that in general $\delta_{\mathrm{BS}}(R)-\rank G+3$ is a dimension function
for $G$, but we cannot prove this now.

Fix $G$ and a dimension function $\delta$  for $G$.
For ideals $\mf a$ and $\mf b$ of a ring $R$ define the following group
\begin{multline*}
S_\delta^{(k)}G(R,\mf a,\mf b)=\{h\in G(R,\mf a\mf b)\mid \ph(h)\in\EE(R',\mf a',\mf b')\\
\text{ for all morphisms }\ph:(R,\mf a,\mf b)\to(R',\mf a',\mf b')\text{ such that }\delta(R')\le k\}.
\end{multline*}

Let $d\ge\delta(R)$. Clearly, $S_\delta^{(d)}G(R,\mf a,\mf b)=\EE(R,\mf a,\mf b)$.
A crucial role in the induction step of the proof of Theorem~\ref{NilpotentThm}
plays the group $S_\delta^{(d-1)}G(R,\mf a,\mf b)$. Here we enlarge this group in the
same style as we have enlarged the relative elementary group.

An element $r$ of a ring $R$ is called $\delta$-regular if $\delta(R/r^kR)<\delta(R)$
for all positive integers $k$.
For an ideal $\mf a$ of a ring $R$ define the group
$$
\wt E_\delta(R,\mf a)=\bigcap_r \Bigl(G(R,r\mf a)\wt E(R,\mf a)\Bigr),
$$
where the intersection is taken over all $\delta$-regular elements $r$.

\begin{lem}\label{tildeEdelta}
Let $\mf a$ and $\mf b$ be ideals of a ring $R$ and $d=\delta(R)>0$.
Then $S_\delta^{(d-1)}G(R,\mf a,\mf b)\le\wt E_\delta(R,\mf a\mf b)$.
\end{lem}

\begin{proof}
If $\delta(R)=\infty$, then using Lemma~\ref{tildeEbig} we get
$$
S_\delta^{(d-1)}G(R,\mf a,\mf b)=S_\delta^{(\infty)}G(R,\mf a,\mf b)=EE(R,\mf a,\mf b)\le
\wt E(R,\mf a\mf b)\le\wt E_\delta(R,\mf a\mf b).
$$

Now, suppose that $\delta(R)<\infty$.
Let $r\in R$ be a $\delta$-regular element.
Put $\bar r=\rho_{\mf a\mf b}(r)$. Let
$\mf q=\{q\in R/\mf a\mf b\mid q{\bar r}^i=0\text{ for some }i\in\N\}$.
It is easy to see that $\mf q$ is an ideal. Since $\delta(R)<\infty$,
by~\ref{DimFunc}(3) $R/\mf a\mf b$ is Noetherian, therefore $\q$
is finitely generated. Hence there exists $k\in\N$ such that
${\bar r}^kq=0\implies {\bar r}^{k-1}q=0$ (this is the same as saying that ${\bar r}^kR/\mf a\mf b$
injects into $(R/\mf a\mf b)_{\bar r}$, cf.~\cite[Lemma~4.10]{BakNonabelian}).
It follows that if $r^kp\in\mf a\mf b$, then $r^{k-1}p\in\mf a\mf b$ and
$r^kp\in r\mf a\mf b$. Therefore, $\mf a\mf b\cap r^kR\subseteq r\mf a\mf b$.

Let $\rho=\rho_{r^kR}$ be the reduction homomorphism. Since $\delta(R/r^kR)\le d-1$,
we have
\begin{align*}
\rho\bigl(&S_\delta^{(d-1)}G(R,\mf a,\mf b)\bigr)\le\EE\bigl(R/r^kR,\rho(\mf a),\rho(\mf b)\bigr)
\text{ and}\\
&S_\delta^{(d-1)}G(R,\mf a,\mf b)\le G(R,r^kR)\EE(R,\mf a,\mf b).
\end{align*}

By definition, $S_\delta^{(d-1)}G(R,\mf a,\mf b)$ and $\EE(R,\mf a,\mf b)$ are contained in
$G(R,\mf a\mf b)$, therefore we may replace $G(R,r^kR)$ by $G(R,r^kR\cap\mf a\mf b)$
in the inclusion above. By the choice of $k$ and by Lemma~\ref{tildeEbig} we have
$$
S_\delta^{(d-1)}G(R,\mf a,\mf b)\le G(R,r\mf a\mf b)\wt E(R,\mf a\mf b).
$$
Since $r$ was chosen arbitrarily, this inclusion implies the result.
\end{proof}

The following result improves the commutator formula from Theorem~\ref{RelThm}.

\begin{lem}\label{nilpotent}
Let $\mf a$ and $\mf b$ be ideals of a ring $R$ and $\delta(R)<\infty$. Suppose that $G$ is simply connected.
Then
$$
[\wt E_\delta(R,\mf a),G(R,\mf b)]\le \EE(R,\mf a,\mf b).
$$
\end{lem}

\begin{proof}
If $\delta(R)\le 1$, then the claim follows from~\ref{DimFunc}(1).
Otherwise, by condition~\ref{DimFunc}(2) we have $G(S^{-1}R,S^{-1}\mf b)=E(S^{-1}R,S^{-1}\mf b)$
for some multiplicative subset $S$ of $R$ consisting of $\delta$-regular elements.
Given $b\in G(R,\mf b)$ we have $\lambda_S(b)\in E(S^{-1}R,S^{-1}\mf b)$. In other
words $\lambda_S(b)$ is a product of elements of the form $b_i=x_{\alpha_i}(t_i/s_i)^{c_i}$, where
$c_i=\prod_j x_{\beta_{ij}}(t_{ij}/s_{ij}$, $t_i\in\mf b$, $t_{ij}\in R$, and
$s_i,s_{ij}\in S$. Let $s$ be the product of all denominators $s_i$ and $s_{ij}$.
Then $b'=\prod_ib_i\in E(R_s,\mf b_s)$. Put $b''=\lambda_s(b)$. Clearly, $b'$ and $b''$
have the same image in $S^{-1}R$. Since $G$ is an affine scheme,
there exist $t\in S$ such that $\lambda_t(b')=\lambda_t(b'')$. Thus, for $r=st$ we have
$\lambda_r(b)=\lambda_t(b'')\in E(R_r,\mf b_r)$.

By Lemma~\ref{relkey} with identity homomorphism $\ph$ there exists a positive integer $m$ such
that $[G(R,r^m\mf a),b]\le\EE(R,\mf a,\mf b)$.
It follows from~\ref{DimFunc}(2) that $r^m$ is $\delta$-regular.
Therefore, $\wt E_\delta(R,\mf a)\le G(R,r^m\mf a)\wt E(R,\mf a)$.
Now, the result follows from Theorem~\ref{RelThm}.
\end{proof}

\begin{cor}\label{IndStep}
Let $\mf a,\mf b,\mf c$ be ideals of a ring $R$. Then
$$
[S_\delta^{(k)}G(R,\mf a,\mf b),G(R,\mf c)]\le S_\delta^{(k+1)}G(R,\mf a\mf b,\mf c).
$$

In particular, if $\delta(R)=0\implies G(R,\mf a)=E(R,\mf a)$ for any pair $(R,\mf a)$,
then conditions~$\ref{DimFunc}(2)$ and~$(3)$ imply condition~$\ref{DimFunc}(1)$.
\end{cor}

\begin{proof}
Let $\ph:R\to R'$ be a ring homomorphism, $\delta(R')=k+1$, $\mf a'=\ph(\mf a)R'$,
$\mf b'=\ph(\mf b)R'$ and $\mf c'=\ph(\mf c)R'$. By Lemma~\ref{tildeEdelta}
$S_\delta^{(k)}G(R',\mf a',\mf b')\le\wt E_\delta(R',\mf a',\mf b')$.
By Lemma~\ref{nilpotent} we have
\begin{multline*}
\ph\bigl([S_\delta^{(k)}G(R,\mf a,\mf b),G(R,\mf c)]\bigr)\le
[S_\delta^{(k)}G(R',\mf a',\mf b'),G(R',\mf c')]\le\\
[\wt E_\delta(R',\mf a'\mf b'),G(R',\mf c')]\le \EE(R',\mf a'\mf b',\mf c').
\end{multline*}
Now the required inclusion follows from the definition of $S_\delta^{(k+1)}G$.

If $G(R,\mf a)=E(R,\mf a)$ provided $\delta(R)=0$, then
$S_\delta^{(0)}G(R,\mf a,\mf b)=G(R,\mf a\mf b)$. We have already noticed that
$S_\delta^{(1)}G(R,\mf a,\mf b)=\EE(R,\mf a,\mf b)$ if $\delta(R)=1$.
Note that we did not use condition~\ref{DimFunc}(1) up to now, thus the inclusion
in the statement follows from~\ref{DimFunc}(2) and~(3).
Now, with $k=0$ and $\mf b=R$ the inclusion turns to condition~\ref{DimFunc}(1).
\end{proof}

\begin{thm}\label{NilpotentThm}
Let $\mf a_0,\dots,\mf a_d$ be ideals of a ring $R$, where $d\ge1$,
and let $\delta$ be a dimension function for $G$.
Then
$$
[G(R,\mf a_0),G(R,\mf a_1),\dots,G(R,\mf a_d)]\le S_\delta^{(d)}G(R,\mf a_0\dots\mf a_{d-1},\mf a_d).
$$
If $\delta(R)\le d$, then
$$
[G(R,\mf a_0),G(R,\mf a_1),\dots,G(R,\mf a_d)]\le\EE(R,\mf a_0\dots\mf a_{d-1},\mf a_d).
$$
\end{thm}

\begin{proof}
We prove the first formula by induction on $d$. If $d\le1$ the claim coincides with condition~\ref{DimFunc}(1).
The induction step follows from Lemma~\ref{IndStep}. As we already noticed after the definition of
the group $S_\delta^{(k)}G(R,\mf a,\mf b)$, if $\delta(R)\le d$, then
$S_\delta^{(d)}G(R,\mf a,\mf b)=\EE(R,\mf a,\mf b)$. Thus, the first formula implies the second one.
\end{proof}

The following straightforward corollary generalizes the results from~\cite{BakNonabelian,HazVav,BHVstrike}
obtained by A.\,Bak, R.\,Haz\-rat, and N.\,Va\-vi\-lov.

\begin{cor}
Let $\mf a_1,\dots,\mf a_d$ be ideals of a Noetherian ring $R$, where $d\ge1$.
Suppose that one of the following holds.
\begin{itemize}
\item
$G$ is simply connected and Bass--Serre dimension of $R$ is not greater than $d$.
\item
$G=\SL_n$, $n\ge3$, and Bass--Serre dimension of $R$ is not greater than $d+n-3$.
\end{itemize}

Then
$
[G(R,\mf a_0),G(R,\mf a_1),\dots,G(R,\mf a_d)]\le\EE(R,\mf a_0\dots\mf a_{d-1},\mf a_d).
$
\end{cor}

\bigskip\bigskip
{\bf Acknowledgement}

\bigskip
The work is partially supported by RFBR projects~11-01-00811,
13-01-00709, 13-01-91150, and 13-01-92699, State Financed research task 6.38.74.2011 at the
St.Petersburg State University and by the Ministry of Education and Science of Russia
(Contracts 114031340002 and 2.136.2014/K).
The idea of this work was obtained in Lahore during visiting Abdus Salam School of
Mathematical Sciences at GC University.


\end{document}